\documentclass[a4paper,10pt]{article}
\pdfoutput=1 
\usepackage[top=1in,bottom=1in,left=0.75in,right=0.75in]{geometry}
\usepackage{amsmath}
\usepackage{amssymb}
\usepackage{amsthm}
\usepackage{mathtools}
\usepackage[only,llbracket,rrbracket]{stmaryrd}
\usepackage[scr=boondoxo]{mathalfa} 
\usepackage[english]{babel}
\usepackage{graphicx}
\usepackage[font=small,figurewithin=section]{caption}
\usepackage{subcaption}
\usepackage{float}
\usepackage[numbers]{natbib} 
\usepackage{hyperref}
\usepackage[all]{xy}
\usepackage{multirow}
\usepackage{array}
\usepackage{enumitem}
\usepackage{tikz}
\usetikzlibrary{matrix,arrows,decorations.pathmorphing}
\usepackage{color}
\usepackage{authblk}
\usepackage[mathscr]{euscript}

\usepackage{relsize}

\numberwithin{equation}{section}
\makeatletter
\g@addto@macro\bfseries{\boldmath}
\makeatother

\hypersetup{colorlinks=true,citecolor=black,filecolor=black,linkcolor=black,urlcolor=black}
\hypersetup{pdfstartview=FitB,pdfpagemode=UseNone}

\setlist{nolistsep}

\newcolumntype{L}[1]{>{\raggedright\let\newline\\\arraybackslash\hspace{0pt}}m{#1}}
\newcolumntype{C}[1]{>{\centering\let\newline\\\arraybackslash\hspace{0pt}}m{#1}}
\newcolumntype{R}[1]{>{\raggedleft\let\newline\\\arraybackslash\hspace{0pt}}m{#1}}
\newcolumntype{N}{@{}m{0pt}@{}}
\SetSymbolFont{stmry}{bold}{U}{stmry}{m}{n}
\usepackage{derivative}

\newcommand{\R}{\mathbb{R}} 

\newcommand{\idmatrix}{\textup{\uppercase\expandafter{\romannumeral 1}}}
\newcommand{\mesh}{\mathcal{T}} 
\newcommand{\bdry}{\partial} 
\newcommand{\tr}{\mathrm{tr}} 

\newcommand{\enr}{\mathrm{enr}} 

\DeclareMathOperator{\grad}{grad}

\let\div\relax 
\DeclareMathOperator{\div}{div}

\newtheoremstyle{boldremark}
    {\dimexpr\topsep/2\relax} 
    {\dimexpr\topsep/2\relax} 
    {}          
    {}          
    {\bfseries} 
    {.}         
    {.5em}      
    {}          
\theoremstyle{definition}
\newtheorem*{definition*}{Definition}
\theoremstyle{plain}
\newtheorem{theorem}{Theorem}
\newtheorem{lemma}{Lemma}
\newtheorem{corollary}{Corollary}
\theoremstyle{boldremark}
\newtheorem{remark}{Remark}
\theoremstyle{definition}
\newtheorem{assumption}{Assumption}

\makeatletter
\newcommand{\dashto}[1][2pt]{
  \settowidth{\@tempdima}{${}\rightarrow{}$}
  \makebox[\@tempdima]{${}\rightarrow{}$}
  \makebox[-\@tempdima]{\hspace{-0.1\@tempdima}\color{white}\rule[0.5ex]{#1}{1pt}}
  \makebox[\@tempdima]{}
  }
\makeatother

\newcommand{\mcO}{\mathcal{O}}

\newcommand{\mcR}{\mathcal{R}}

\newcommand{\bsi}{\boldsymbol{i}}
\newcommand{\bsj}{\boldsymbol{j}}

\newcommand{\bsq}{\boldsymbol{q}}

\newcommand{\bsv}{\boldsymbol{v}}
\newcommand{\bsw}{\boldsymbol{w}}

\newcommand{\bsL}{\boldsymbol{L}}

\newcommand{\bsS}{\boldsymbol{S}}

\newcommand{\scB}{\mathscr{B}}
\newcommand{\scC}{\mathscr{C}}

\newcommand{\scM}{\mathscr{M}}

\newcommand{\scU}{\mathscr{U}}
\newcommand{\scV}{\mathscr{V}}
\newcommand{\scW}{\mathscr{W}}

\newcommand{\nml}{\mathbf{n}}

\usepackage{mathtools}

\makeatletter
\DeclareRobustCommand\widecheck[1]{{\mathpalette\@widecheck{#1}}}
\def\@widecheck#1#2{%
    \setbox\z@\hbox{\m@th$#1#2$}%
    \setbox\tw@\hbox{\m@th$#1%
       \widehat{%
          \vrule\@width\z@\@height\ht\z@
          \vrule\@height\z@\@width\wd\z@}$}%
    \dp\tw@-\ht\z@
    \@tempdima\ht\z@ \advance\@tempdima2\ht\tw@ \divide\@tempdima\thr@@
    \setbox\tw@\hbox{%
       \raise\@tempdima\hbox{\scalebox{1}[-1]{\lower\@tempdima\box
\tw@}}}%
    {\ooalign{\box\tw@ \cr \box\z@}}}
\makeatother

\newcommand{\bfT}{\mathbf{T}}

\DeclareMathOperator{\meas}{meas}

\newcommand{\tend}{t_\text{end}}
\newcommand{\interface}{\Gamma_{\textrm{int}}}

\newcommand{\iapp}{I_{\textrm{app}}}

\newcommand{\ibv}{I_{\textrm{BV}}}
\newcommand{\kbv}{k_{\textrm{BV}}}
\newcommand{\csmax}{c_{s,\text{max}}}
\newcommand{\ocp}{\Phi_{s,\textrm{open}}}
\newcommand{\ocpsa}{\Phi_{sa,\textrm{open}}}

\newcommand{\Gpos}{\Gamma_{\text{cc}+}}
\newcommand{\Gneg}{\Gamma_{\text{cc}-}}
\newcommand{\Gtop}{\Gamma_{\text{top}}}
\newcommand{\Gbot}{\Gamma_{\text{bottom}}}

\newcommand{\prev}{{\textrm{prev}}}
\newcommand{\dt}{\Delta t}

\usepackage{tikz}
\usetikzlibrary{shapes.geometric, shapes.misc, arrows}
\tikzstyle{startstop} = [rectangle, rounded corners, minimum width=1.8cm, minimum height=0.5cm,text centered, draw=black, fill=red!30]
\tikzstyle{io} = [trapezium, trapezium left angle=70, trapezium right angle=110, minimum width=1.8cm, minimum height=0.5cm, text centered, draw=black, fill=blue!30]
\tikzstyle{process} = [rectangle, minimum width=1.8cm, minimum height=0.5cm, text centered, draw=black, fill=orange!30]
\tikzstyle{decision} = [diamond, minimum width=1.8cm, minimum height=0.5cm, text centered, draw=black, fill=green!30]
\tikzstyle{forloop} = [chamfered rectangle, chamfered rectangle xsep = 0.5cm, text centered, draw=black, fill=teal!30]
\tikzstyle{cont} = [circle, minimum width=0.3cm,draw=black,fill=teal!30]
\tikzstyle{arrow} = [thick,->,>=stealth]

\usepackage{listings}
\usepackage{lscape}
\usepackage{xfrac}
\usepackage{xcolor}
\usepackage{cancel}

\def\be{\begin{equation}}
\def\ee{\end{equation}}
\def\ba{\begin{array}}
\def\ea{\end{array}}
\def\bea{\begin{eqnarray}}
\def\eea{\end{eqnarray}}
\def\beas{\begin{eqnarray*}}
\def\eeas{\end{eqnarray*}}

\begin{document}
%
\title{Formulation and analysis of a DPG discretization for a simplified electrochemical model}
\author[1]{Jaime Mora-Paz\thanks{e-mail: jaimed.morap@konradlorenz.edu.co}}
\affil[1]{Fundación Universitaria Konrad Lorenz, Bogota DC, Colombia}
\date{\vspace{-10mm}}
%
\maketitle
\renewcommand{\abstractname}{\large Abstract}
\begin{abstract}
\small 
\noindent
We present a simplified model consisting on two linear elliptic boundary-value problems that represent a single step and single fixed-point iteration in an electrochemical battery model. The main variables are the concentration and the electric potential, whose equation is assigned a Robin BC with a very important physical interpretation. The solvability of both equations is studied in different funcional settings, to finally prove the well-posedness of a broken mixed variational formulation. The latter formulation opens the opportunity of performing discretization and numerical solution via the Discontinuous Petrov-Galerkin (DPG) method, which guarantees discrete stability thanks to optimal test functions. With only the usual assumptions on the data and the discretization, we show that the method herein proposed is convergent. This analytical effort complements other recent works on batttery multiphysics, that have been more modeling and computing oriented, and establishes the DPG approach as a valuable tool to contribute toward an efficient analysis and design of batteries.
\end{abstract}
\section{Introduction}
\label{sec:Introduction}
Two appearingly independent boundary value problems (BVP) are next presented.
First, what we call the \emph{concentration problem}, since the unknown $c:\Omega\to \R$ represents
the volumetric concentration of some species in a medium extended over $\Omega\subset\R^d$, where $d\in\{2,3\}$. 
The parameters will be explained below, but as is we recognize the form of a reaction-diffusion PDE with Neumann boundary condition (BC).
\begin{equation}
    \left\{
    \begin{array}{rll}
         c- \dt D\div\grad c &=c_\prev  &\text{in }\Omega  \vspace{0.2cm}\\
        (-D\grad c)\cdot\nml &=J    & \text{on }\bdry\Omega
    \end{array}
    \right.
    \label{eq:conce_order2_strong}
\end{equation}
In second place, we introduce the \emph{potential problem}, where we seek for the scalar field $\phi:\Omega\to\R$,
that stands for an electric potential distributed over the same $\Omega$. Contrary to the first problem, 
the potential equation has the form of a Poisson equation (or Laplace equation, if $\div\bsS$ vanishes) with mixed boundary conditions, but we want to stress on the last one, a Robin BC. In this problem the boundary is decomposed in disjoint subsets  $\Gamma_D,\; \Gamma_N,\; \Gamma_R \subset \bdry\Omega$, such that$\bdry\Omega = \overline{\Gamma_D \cup \Gamma_N \cup \Gamma_R}$.
\begin{equation}
    \left\{
    \begin{array}{rll}
                    -\kappa\div\grad \phi &= \div\bsS & \text{in }\Omega  \vspace{0.2cm}  \\
                                     \phi &= 0 & \text{on }\Gamma_D  \\
              (-\kappa\grad\phi)\cdot\nml &= I & \text{on }\Gamma_N  \\
    (-\kappa\grad\phi)\cdot\nml-\beta\phi &= R & \text{on }\Gamma_R
    \end{array}
    \right.
    \label{eq:poten_order2_strong}
\end{equation}

The two problems are already in the form of a second order linear PDE. However, it is also frequent to deal with the corresponding diffusive fluxes. Vector field $\bsj$ will be called the species flux, while $\bsi$ is the current density, and their respective constitutive laws are given by
\begin{align}
    \bsj &= -D\grad c \; ,  \label{eq:ionflux_def}\\
    \bsi &= -\kappa \grad \phi - \bsS \;.  \label{eq:currdens_def}
\end{align}
Whenever the constitutive laws are not incorporatted into the equations of the first principles (conservation of mass and Gauss's law), we then deal with first-order systems. The concentration problem in the form of a first-order BVP is presented in \eqref{eq:conce_order1_strong}, whereas the potential problem is reformulated as in \eqref{eq:poten_order1_strong}.
\begin{equation}
    \left\{
    \begin{array}{rll}
            c + \dt\div\bsj &= c_\prev  & \text{in }\Omega  \\
        D^{-1}\bsj +\grad c &=0        & \text{in }\Omega  \vspace{0.2cm}  \\
             \bsj\cdot \nml &=J        & \text{on }\bdry\Omega
    \end{array}
    \right.
    \label{eq:conce_order1_strong}
\end{equation}
\begin{equation}
    \left\{
    \begin{array}{rll}
                        \div \bsi &= 0                & \text{in }\Omega  \\
        \kappa^{-1}\bsi+\grad\phi &= -\kappa^{-1}\bsS & \text{in }\Omega  \vspace{0.2cm}  \\
                             \phi &= 0                & \text{on }\Gamma_D  \\
                    \bsi\cdot\nml &= I                & \text{on }\Gamma_N  \\
          \bsi\cdot\nml-\beta\phi &= R                & \text{on }\Gamma_R
    \end{array}
    \right.
    \label{eq:poten_order1_strong}
\end{equation}

\subsection{Motivation}
Where do the concentration and potential problems come from? For the sake of the present discussion, we consider the electrochemical model for a lithium-ion battery microstructure. In it, three different materials are at interplay: the negative electrode or anode (subscript $sa$), the positive electroode or cathode (subscript $sc$) and the electrolytic medium or electrolyte (subscript $e$). In each material or medium, different concentration fields and potential fields are accounted for. The spatial distribution, time evolution and interacion of those fields is governed by the following initial-boundary value problem with six unknowns.
\begin{equation}
\left\{
    \begin{array}{rll}
        \dot{c_{sa}}
        - \div \left(D_{sa}\grad c_{sa}\right)
            & = 0
            & \text{ in }\Omega_{sa} \times(0,\tend] ,   \\
        \dot{c_{sc}}
        - \div \left(D_{sc}\grad c_{sc}\right)
            & = 0
            & \text{ in }\Omega_{sc} \times(0,\tend] ,   \\
        \dot{c_e}
        - \div \left(D_e\grad c_e\right)
            & = 0
            & \text{ in }\Omega_e \times(0,\tend] ,  \\
        -\div \left(\sigma_{sa}\grad \phi_{sa}\right) 
            &= 0
            & \text{ in }\Omega_{sa} \times(0,\tend] ,  \\
        -\div \left(\sigma_{sc}\grad \phi_{sc}\right) 
            &= 0
            & \text{ in }\Omega_{sc} \times(0,\tend] ,  \\
        -\div\left(\kappa_e\grad\phi_e\right) 
            & = \div\left(\kappa_D\grad \ln c_e\right)
            & \text{ in }\Omega_e \times(0,\tend] .
    \end{array}
\right.      
\label{eq:pde_system}
\end{equation}
Of course, the PDEs are accompanied by appropriate boundary and initial conditions. Some of that is discussed below. Because of the difference in material properties, in the constitutive laws for the fluxes we must specify both the variable and diffusivity coefficient in question:
\begin{align*}
    \bsj_{sa} &= -D_{sa}\grad c_{sa} \; ,  \\
    \bsj_{sc} &= -D_{sc}\grad c_{sc} \; ,  \\
    \bsj_{e}  &= -D_{e} \grad c_{e} \; ,  \\
    \bsi_{sa} &= -\sigma_{sa} \grad \phi_{sa} \; ,  \\
    \bsi_{sc} &= -\sigma_{sc} \grad \phi_{sc} \; ,  \\
    \bsi_{e}  &= -\kappa_{e} \grad \phi_{e} - \kappa_{D} \grad \ln c_{e} \,.
\end{align*}
Despite the number of equations herein proposed, notice that all of the $\phi$ and $\bsi$ equations can be accomodated into the \emph{potential} problem above (whether as a second order PDE or a first-order system). As for the \emph{concentration} problem, we regard the model \eqref{eq:conce_order2_strong} as a backward Euler approximation in time, so that the dotted term is replaced by $(c-c_{\prev})/\dt$, where $\Delta t$ is a small number standing for the time integration's step length. As seen in \cite{bai_chemo-mechanical_2021,miranda_effect_2019,gatica_analysis_2018,gritton_using_2017} several of those diffusivity constants are variable, sometimes in a non-linear coupled fashion. We do not delve into that possibility here, but we mention it because the interaction among variables may require a fixed-point iterative approach. If that's the case, the model problems for concentration and potential are somehow the idealization of one iteration, within one time integration step. Moreover, it is possible to solve the problem in a staggered way: first the time dependent equation (concentration), later the time independent one (potential), as in \cite{mora2024high}. That is why, in the model problems \eqref{eq:conce_order1_strong}, \eqref{eq:poten_order1_strong}, \eqref{eq:conce_order2_strong} and \eqref{eq:poten_order2_strong} we are free from non-linearities, because we want to analyze the mere application of variational formulations on that very specific steady and linear instance of an otherwise dynamic and nonlinear phenomenos, and the finite element discretizations for their numerical solution, with the most general possible conditions on the material parameters and the loads.

However, the coupling among materials must be added to the model, because the interfacial electrochemical reaction is what drives all the physical processes within a battery that is being discharged or charged.
\subsubsection{Butler-Volmer Reaction Kinetics}
\label{subsec:butler-volmer}
The electrochemical reaction that arises on the electrode-electrolyte interface is the main mechanism that couples the physical processes occurring in the electrodes to those in the electrolyte. If there is certain electric potential difference across the interface between one solid electrode and the electrolytic solution, an electrochemical reaction occurs, so producing transport of electrically charged particles (Lithium ions) from one side to another. That transport induces an electric current density normal to the surface, denoted $\ibv$. The Butler-Volmer formula models the magnitude of the current, depending on the concentrations and potentials of both media:
\begin{equation}
    \begin{aligned}
    \ibv & = \frac{I_c F}{R T} \eta \ , \\
    I_c & = \kbv F\, c_e^{\sfrac{1}{2}}\left(\csmax-c_s\right)^{\sfrac{1} {2}}c_s^{\sfrac{1}{2}}\ ,
\end{aligned}
\label{eq:butler-volmer} 
\end{equation}
where $F$ is Faraday's constant, $R$ is the universal gas constant, $T$ is the absolute temperature (assumed to be constant),
$I_c$ is known as the exchange current density, $\kbv$ is a specific reaction constant, $\csmax$ is the saturation or maximum concentration possible at each electrode, and $\eta$ represents the overpotential:
\begin{equation}
    \eta=\phi_s-\phi_e-\ocp(\check{c}_s) \,.
    \label{eq:overpotential}
\end{equation}
In \eqref{eq:overpotential}, $\ocp$ denotes the so called open circuit potential, an electrode's material property that varies depending on the local \emph{state of charge} (see \cite{doyle1993modeling} for examples of experimentally fitted functions for the open-circuit potential), 
\begin{equation}
    \check{c}_s=c_s/\csmax.
    \label{eq:stateofcharge}
\end{equation}

The dependence on the potentials $\phi_e$ and $\phi_s$ in \eqref{eq:butler-volmer} has been linearized, as explained in \cite{mora2024high}, because it originally portrays exponential functions or a hyperbolic sine function \cite{linden2002handbook,trembacki_fully_2015}. The computation of $\ibv$, $I_c$, $\eta$, $\ocp$ and $\check{c}$ is carried out using the traces of $c_s$, $c_e$, $\phi_s$, $\phi_e$ on $\interface$. 

Butler-Volmer equation is also useful for determining the species flux produced by the reaction, which is  proportional to $\ibv$ through Faraday's constant $F$.

\subsubsection{Boundary and interface conditions in the original electrochemical model}
Figure \ref{fig:boundaries-diagram} shows the parts of the boundary where the BC and the interface conditions are to be imposed. Since we have assumed that the domain is a small representative part of the battery structure, it is reasonable to prescribe symmetry-like BC on some portions of $\Gamma$.
\begin{figure}
    \centering
    \includegraphics[width=0.35\textwidth]{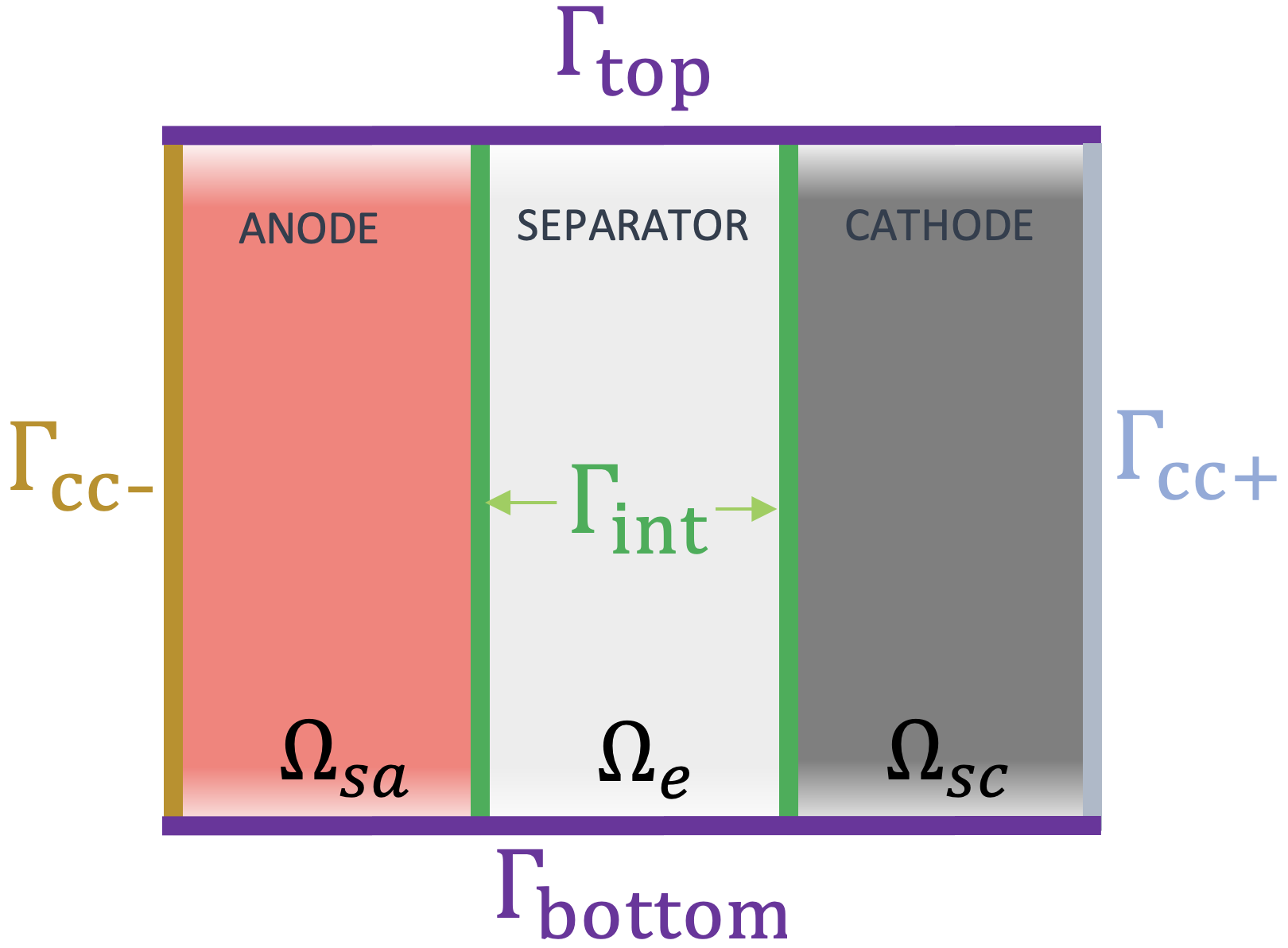}
    \caption{Battery diagram showing boundary portions that are referenced in the specification of boundary conditions.}
    \label{fig:boundaries-diagram}
\end{figure}
\paragraph{Boundary, interface and initial conditions for concentration in the anode}.
\begin{equation}
\left\{ 
\begin{array}{rll}
    \bsj_{sa}\cdot\nml &= 0       &\text{ on }(\bdry\Omega_{sa}\cap \Gtop) \cup (\bdry\Omega_{sa}\cap \Gbot)  \\
    \bsj_{sa}\cdot\nml &= \ibv/F  &\text{ on }(\bdry\Omega_{sa}\cap \Gamma_{e}) \cup (\bdry\Omega_{sa}\cap \Gamma_{e}) \\
    c_{sa}(\cdot,0)    &= c_{e,0} &\text{ in }\Omega_{sa}\times\{0\}
\end{array}
\right.
\end{equation}
\paragraph{Boundary, interface and initial conditions for concentration in the cathode}.
\begin{equation}
\left\{ 
\begin{array}{rll}
    \bsj_{sc}\cdot\nml &= 0       &\text{ on }(\bdry\Omega_{sc}\cap \Gtop) \cup (\bdry\Omega_{sc}\cap \Gbot)  \\
    \bsj_{sc}\cdot\nml &= \ibv/F  &\text{ on }(\bdry\Omega_{sc}\cap \Gamma_{e}) \cup (\bdry\Omega_{sc}\cap \Gamma_{e}) \\
    c_{sc}(\cdot,0)    &= c_{e,0} &\text{ in }\Omega_{sc}\times\{0\}
\end{array}
\right.
\end{equation}
\paragraph{Boundary, interface and initial conditions for concentration in the electrolyte}.
\begin{equation}
\left\{ 
\begin{array}{rll}
    \bsj_{e} \cdot\nml &= 0             &\text{ on }(\bdry\Omega_{e}\cap \Gtop) \cup (\bdry\Omega_{sc}\cap \Gbot) \\
    \bsj_{e} \cdot\nml &=-(1-t_+)\ibv/F &\text{ on }(\bdry\Omega_{e}\cap \Gamma_{sa}) \cup (\bdry\Omega_{e}\cap \Gamma_{sc}) \\
    c_{e} (\cdot,0)    &= c_{e,0}       &\text{ in }\Omega_{e}\times\{0\}
\end{array}
\right.
\end{equation}
\paragraph{Boundary and interface conditions for potential in the anode}.
\begin{equation}
\left\{ 
\begin{array}{rll}
    \phi_{sa}          &=0    &\text{ on }\Gneg \\
    \bsi_{sa}\cdot\nml &=0    &\text{ on }(\bdry\Omega_{sa}\cap \Gtop) \cup (\bdry\Omega_{sa}\cap \Gbot) \\
    \bsi_{sa}\cdot\nml &=\ibv &\text{ on }(\bdry\Omega_{sa}\cap \bdry\Omega_e)
\end{array}
\right.
\end{equation}
\paragraph{Boundary and interface conditions for potential in the cathode}.
\begin{equation}
\left\{ 
\begin{array}{rll}
    \bsi_{sc}\cdot\nml &=\iapp &\text{ on }\Gpos \\
    \bsi_{sc}\cdot\nml &=0     &\text{ on }(\bdry\Omega_{sc}\cap \Gtop) \cup (\bdry\Omega_{sc}\cap \Gbot) \\
    \bsi_{sc}\cdot\nml &=\ibv  &\text{ on }(\bdry\Omega_{sc}\cap \bdry\Omega_e)
\end{array}
\right.
\end{equation}
\paragraph{Boundary and interface conditions for potential in the electrolyte}.
\begin{equation}
\left\{ 
\begin{array}{rll}
    \bsi_{e}\cdot\nml &=0     &\text{ on }(\bdry\Omega_{e}\cap \Gtop) \cup (\bdry\Omega_{sc}\cap \Gbot) \\
    \bsi_{e}\cdot\nml &=-\ibv &\text{ on }(\bdry\Omega_{e}\cap \Gamma_{sa}) \cup (\bdry\Omega_{e}\cap \Gamma_{sc})
\end{array}
\right.
\end{equation}

\subsubsection{Connection with the model potential problem}
The Robin BC in the model problem \eqref{eq:poten_order1_strong} is obtained from the incorporation of Butler-Volmer's equation, as it relates the current density and the electric potential. All the other terms are moved to the load. Next, we show the correspondence between the symbols in the model problem and, for example, the actual anode equation.
\begin{equation*}
    \bsi_{sa}\cdot\nml - \underbrace{ \frac{I_c F}{R T} }_{\beta}\phi_{sa} = 
    \underbrace{\frac{I_c F}{R T}(-\phi_{e} - \ocpsa(\check{c}_{sa}) )}_{R}
\end{equation*}
When solving the actual complex system with a numerical method, the usage of fixed-point iterations lets us focus on each subdomain at a time. That is why we have simplified the equations to the above called \emph{concentration} and \emph{potential} problems. That reductioon does not mean, however, that we are dealing with a trivial problem, but with an important stepping stone in the way to understand and possibly construct more strategies for solving the multiphysical models of rechargeable batteries.

\subsection{General assumptions and notation}
The name \emph{domain} means exactly a connected, open and bounded subset in $\R^d$. We restrict the functional spaces to domains with Lipschitz or piecewise smooth boundaries, and they contain real-valued functions only. For the potential problem, $\Gamma_D$ may be empty (if not, then we also assume $\meas(\Gamma_D)>0$), but both $\Gamma_R$ and $Gamma_N$ must be non-empty and with positive measure.

The inner product of $L^2(\Omega)$ scalar fields, or $\bsL^2(\Omega)$ vector fields over certain (Lipschitz) subdomain $\omega\subseteq\Omega$ is denoted $(\cdot,\cdot)_\omega$. The corresponding norm is denoted $\|\cdot\|_{\omega}$, both for scalars or vectors. We drop the domain subscript when there is no place for ambiguity. In what follows, we work with the standard $H^1(\omega)$ norm,
\begin{equation*}
\| u \|^2_{1}:=\| u \|^2 + \| \grad u \|^2.
\end{equation*}
An auxiliary space for us will be the $H(\div,\omega)$ space, which contains vector fields whose divergence lies in $L^2$, and have norm
\begin{equation*}
\|V\|_{\div,\omega} := \| \|^2 + \| \|^2.
\end{equation*}
Let $\gamma$ be a subdomain boundary, or part of it, i.e., $\gamma\subseteq\bdry\omega$. We define the pairing $\langle\cdot,\cdot\rangle_{\gamma}$ as the $L^2(\gamma)$ inner product, that it
$$
\langle u,v\rangle_{\gamma}:=\int_\gamma u\,v\, ds.
$$
The application of this pairing on functions defined on the domain means that we act on their restritions to $\gamma$. We require additional spaces defined on boundaries.
The $H^{1/2}(\gamma)$ space, where $\gamma$ is a subdomain boundary, or part of it, $\gamma\subseteq\bdry\omega$, is endowed with the so called \emph{minimum extension norm}:
\begin{equation*}
    \| u \|_{1/2,\gamma}:=\inf\{\|U\|_{1}:\ U\in H^1(\omega),U|_{\gamma}=u \}.
\end{equation*}
The $H^{-1/2}(\gamma)$ space comes with the norm
\begin{equation*}
    \| v \|_{-1/2,\gamma}:=\inf\{\|V\|_{\div,\omega}:\ V\in H(\div,\omega),(V\cdot\nml)|_{\gamma}=v \}.
\end{equation*}
Notice that, if $\gamma$ is an entire boundary, we can see $H^{-1/2}(\gamma)$ and $H^{1/2}(\gamma)$ as dual spaces of each other \cite{McLeanSobolev}. In that case, we extend $\langle\cdot,\cdot\rangle_{\gamma}$ as the duality pairing for $H^{1/2}(\gamma) \times H^{-1/2}(\gamma)$ \cite[\S 5.3.3]{demkowicz2023mathematical}. It this is the case, an equivalent norm is therefore the operator norm
\begin{equation*}
    \| v \|_{-1/2,\gamma}:=
    \sup\limits_{\begin{array}{c} u\in H^{1/2}(\gamma)\\\| u \|_{1/2,\gamma}=1 \end{array}}
     \langle u,v \rangle_{\gamma}.
\end{equation*}

\subsection{Goal, scope and outline of this paper}
In order to understand the behaviour of different variational formulations applied to the model concentration and potential problems introduced above, we apply fundamentals of functional analysis to get well-posedness theorems on different functional settings. Ultimately, one of these formulations enables the use of a poweful finite element methodology, known as DPG, which is stable by design, and a flexible approach with respect to the functional spaces in play. In the future, DPG can provide us with the capacity of automatic adaptive meshing, so that we can apply the proposed discrete setting to complicated battery simulations, once it gets integrated into a dynamic nonlinear solver.

We explain what the reader will find in this paper. In section 2, we formulate the concentration and potential problems in infinite-dimensional Hilbert spaces, and achieve well-posedness results. In section 3, some preliminaries on broken Sobolev spaces are brought over, so that we can apply them to show the solvability of broken mixed variational formulations for each model. The fourth section explains all the main considerations to be able to successfully discretize the broken formulations, and attain a stable convergent finite element method. Section 5 closes this work with some final remarks.
\section{Well-posedness of the continuous problem}

\subsection{Classical weak formulation}

We begin by establishing the well-posedness (i.e., solution's existence, uniqueness and continuous dependence on the data) of the classical variational formulation for the concentration problem \eqref{eq:conce_order2_strong}.
\begin{equation}
\left\{
\begin{array}{ll}
\text{Find }c\in H^1(\Omega) \text{ such that:}      \\
        \left(c,r\right)_\Omega + \dt\left(D\grad c,\grad r\right)_\Omega \;=\; 
        \left(c_\prev,r\right)_\Omega - \dt \langle J, r \rangle_{\bdry\Omega}
        \qquad \text{for every } r\in H^1(\Omega)   
\end{array}
\right.
\label{eq:conce_classical_weak}
\end{equation}
In order to prove that \eqref{eq:conce_classical_weak} is well-posed, we assume a handful of facts.
\begin{assumption}    
Coefficients $D$ and $\dt$ are positive parameters, sufficiently small ($\dt, D < 1$). Loads $c_\prev$ and $J$ belong to $H^1(\Omega)\subset L^2(\Omega)$ and $H^{-1/2}(\bdry\Omega)$, respectively.
\end{assumption}
Despite its simplicity, we state and prove Lemma \ref{thm:conce_classical_well-posed} for completeness.
\begin{lemma}
    \label{thm:conce_classical_well-posed}
    Let Assumption 1 hold. Then, problem \eqref{eq:conce_classical_weak} is well-posed and for its unique solution $c\in H^1(\Omega)$ there holds
    \begin{equation*}
        \| c \|_{1} \leq \frac{1}{D\dt} \|c_{prev}\| + \frac{1}{D} \| J \|_{H^{-1/2}} \;.
    \end{equation*}
\end{lemma}
\begin{proof}
    As the problem is fully linear, symmetric and formulated on Hilbert spaces, the result is attained by a straightforward application of Lax-Milgram's theorem \cite[\S 6.2.1]{evans2010partial}. Cauchy-Schwartz inequality together with the assumptions on the parameters easily produce
    \begin{align*}
        \dt D \| u \|^2_{1} \leq \left(u,u\right)_\Omega + \dt\left(D\grad u,\grad u\right)_\Omega\;, \\
        |\left(u,r\right)_\Omega + \dt\left(D\grad u,\grad r\right)_\Omega| \leq \| u \|_{1} \| r \|_{1} \;,\\
        |\left(c_\prev,r\right)_\Omega - \dt \langle J, r \rangle_{\bdry\Omega}|\leq 
        \left(\| c_\prev\| +\dt \| J \|_{H^{-1/2}}  \right) \|r\|_{1}\,.
    \end{align*}
    The first inequality represents the coercivity of the bilinear form on $H^1(\Omega)$, the second bound shows its continuity on $H^1(\Omega)\times H^1(\Omega)$, and the last one is the continuity of the linear functional on $H^1(\Omega)$. Given these three conditions, Lax-Milgram ensures that the solution to the weak formulation \eqref{eq:conce_classical_weak} exists, that it is unique, and that it depends continuously on the data precisely through the given estimate.
\end{proof}

The classical variational formulation for the potential problem \eqref{eq:poten_order2_strong} reads
\begin{equation}
\left\{
\begin{array}{ll}
\text{Find }\phi\in H^1_{\Gamma_D}(\Omega) \text{ such that:}      \\
\left(\kappa\grad \phi,\grad \zeta\right)_\Omega + \langle \beta\phi, \zeta \rangle_{\Gamma_R}\;=\; 
    -\left(\bsS,\grad \zeta\right)_\Omega + \langle R, \zeta \rangle_{\Gamma_R} + \langle I, \zeta \rangle_{\Gamma_N}
        \qquad \text{for every } \zeta\in H^1_{\Gamma_D}(\Omega)   
\end{array}
\right.
\label{eq:poten_classical_weak}
\end{equation}
where we have introduced the subspace $H^1_{\Gamma_D}(\Omega)=\{u\in H^1(\Omega):\ u|_{\Gamma_D}=0\}$. The relevant assumptions on the data are presented next.
\begin{assumption} Coefficient  $\kappa$ is a positive constant, whereas $\beta$ may be a function on $\Gamma_R$ , positive, bounded and non-degenerate, namely
$$
\beta\in L^\infty(\Gamma_R),\ 0<\beta_{\inf} :=  \inf\limits_{\Omega} \beta,\  \beta_{\sup}:= \sup\limits_{\Omega} \beta  < +\infty; 
$$
$\bsS$ is an $L^2$ vector field in $\Omega$, and the Robin and Neumann loads satisfy $R\in L^\infty(\Gamma_R)$, and $I\in L^\infty(\Gamma_N)$, respectively.
\end{assumption}
\begin{lemma}
    \label{thm:poten_classical_well-posed}
    Let Assumption 2 hold. Then, problem \eqref{eq:poten_classical_weak} is well-posed and for its unique solution $\phi\in H^1_{\Gamma_D}(\Omega)$ there holds
    \begin{equation*}
        \| \phi \|_{1} \leq \frac{M}{\alpha_{\phi}} \left( \|\bsS\| + \| R \|_{H^{-1/2}} + \| I \|_{H^{-1/2}} \right)
    \end{equation*}
\end{lemma}
    for some positive constants $M,\alpha_{\phi}$.
\begin{proof}
    Let $b^{\scC}_{\phi}(\phi,\zeta)$ and $\ell(\zeta)$ be the bilinear form on the left-hand side of \eqref{eq:poten_classical_weak}, and the linear functional on the right-hand side, respectively. The latter can be easily bounded by 
    $$
    \left( \|\bsS\| + \| R \|_{H^{-1/2}} + \| I \|_{H^{-1/2}} \right)\|\phi\|_H^1
    $$
    just with Cauchy-Schwartz's inequality and the fact that the $H^{-1/2}$ norm is indeed an operator norm on $H^1$ functions \cite{BrokenForms15}.
    
    Regarding $b^{\scC}_{\phi}(\cdot,\cdot)$, it is clear that it is continuous on $H^1_{\Gamma_D}(\Omega) \times H^1_{\Gamma_D}(\Omega)$, with bounding constant $M=\max\{\kappa,\beta_{\sup}\}$. 
    We claim that it is also coercive, that is, that there is some $\alpha$ such that 
    $\alpha_{\phi} \| \phi \|^2_{1}\leq b(\phi,\phi)$ for any $\phi\in H^1_{\Gamma_D}(\Omega)$. 
    The argument, which goes by contradiction, is standard and uses compact embeddings in Sobolev spaces.

    Suppose there is no such constant. Then, for any $n>0$ there exists $u_n\in H^1_{\Gamma_D}$ such that $\| u \|^2_{1}> n\, b^{\scC}_{\phi}(u_n,u_n)$. Let us then construct a sequence $\{v_n\}_{n=1}^\infty$ of functions with unit norm by
    $$
    v_n = \tfrac{u_n}{\| u \|_{1}}.
    $$
    Notice that $b^{\scC}_{\phi}(v_n,v_n)<1/n$, which implies $\beta_{\inf}\|v_n\|^2_{L^2(\Gamma_R)}<1/n$ and $\kappa\|\grad v_n\|^2<1/n$, by definition of $b$. Since $\{v_n\}$ is a bounded sequence, by the Rellich-Kondrachov theorem \cite[\S 5.7]{evans2010partial}, there exists a convergent subsequence $v_{n_k}\to v$ in $L^2(\Omega)$. Per the above observations, we have $\grad v_{n_k} \to 0$ strongly in $L^2(\Omega)$.
    For any smooth test vector field $\zeta\in  C^\infty_0(\Omega;\R^d)$ there holds
    $$
     \int_\Omega  v \div\zeta dx =
     \lim\limits_{{n_k}\to\infty} \int_\Omega  v_{n_k} \div\zeta dx = 
    -\lim\limits_{{n_k}\to\infty} \int_\Omega \grad v_{n_k}\cdot\zeta dx = 
    0,
    $$
    which means that $\grad v_{n_k}\to \grad v$ in $L^2$. Consequently, $v\in H^1$ and since all the $v_{n_k}$ have unit norm, $\| v \|_{1}=1$ too. However, we see that $\grad v = 0$, meaning that $v$ is constant, and because of the continuity of the trace $H^1\to H^{1/2}(\Gamma_R)\hookrightarrow L^2(\Gamma_R)$, the above result $\|v_{n_k}\|_{L^2(\Gamma_R)}\to 0$ implies that $v|_{\Gamma_R}=0$. The only possibility is that $v\equiv 0$, but this contradicts the fact that $v$ has unit norm. Thus the coercivity is proven, and again by Lax-Milgram, we obtain the theorem's bound.
\end{proof}

\begin{remark}
    In Lemma \ref{thm:poten_classical_well-posed}, the most important aspect is that the result holds even when $\Gamma_D=\varnothing$. If that were not the case, it would facilitate the coercivity proof by bringing over a more usual Poincaré's inequality.
\end{remark}

\subsection{Mixed formulations}
By relaxing the balance equation of problem \eqref{eq:conce_order1_strong}, we can get the following mixed variational formulation.
\begin{equation}
\left\{
\begin{array}{ll}
\text{Find }c\in H^1(\Omega),\;\bsj\in \bsL^2(\Omega) \text{ such that:}      \\
        \left(c,r\right)_\Omega - \dt\left(\bsj,\grad r\right)_\Omega \;=\; 
        \left(c_\prev,r\right)_\Omega - \dt \langle J, r \rangle_{\bdry\Omega}
        \qquad \text{for every } r\in H^1(\Omega)    \\
        \left(\grad c,\bsw\right)_\Omega + \left(D^{-1}\bsj, \bsw\right)_\Omega \;=\; 
        0  \qquad \text{for every } \bsw\in \bsL^2(\Omega)
\end{array}
\right.
\label{eq:conce_mixed_weak}
\end{equation}

Similarly, in the potential equation, from \eqref{eq:poten_order1_strong} we can get to
\begin{equation}
\left\{
\begin{array}{ll}
\text{Find }\phi\in H^1_{\Gamma_D}(\Omega),\;\bsj\in \bsL^2(\Omega) \text{ such that:}      \\
        \langle \beta \phi,\zeta \rangle_{\Gamma_R} - \left(\bsi,\grad \zeta\right)_\Omega \;=\; 
        - \langle I, \zeta \rangle_{\Gamma_N} - \langle R, \zeta \rangle_{\Gamma_R}
        \qquad \text{for every } \zeta\in H^1_{\Gamma_D}(\Omega)    \\
        \left(\grad \phi,\bsv\right)_\Omega + \left(\kappa^{-1}\bsi, \bsv\right)_\Omega \;=\; 
        -\left(\kappa^{-1}\bsS,\bsv\right)_\Omega  \qquad \text{for every } \bsv\in \bsL^2(\Omega)
\end{array}
\right.
\label{eq:poten_mixed_weak}
\end{equation}
In the case of the mixed formulations, in addition to uniqueness and existence of $(c,\bsj)$ or $(\phi,\bsi)$, the well-posedness comprises being able to bound their norms in $H^1(\Omega) \times \bsL^2(\Omega)$. For this purpose, it turns quite helpful to move to a more general framework than coercivity. We prefer to incorporate inf-sup conditions on the bilinear forms $b:\scU\times\scV\to\R$, for Hilbert spaces $\scU$ (trial) and $\scV$ (test), that is:
\begin{equation}
    \exists \alpha>0:\ \alpha\| u\|_{\scU} \leq \sup\limits_{v\in\scV\setminus\{0\}}\frac{b(u,v)}{\|v\|_{\scV}},\ \forall u\in\scU.
    \label{eq:infsup_abstract}
\end{equation}
By the Babu\v{s}ka-Ne\v{c}as theorem \cite{DemkowiczClosedRange}, complying with this condition delivers well-posedness for the abstract variational problem
\begin{equation}
    \left\{ \begin{array}{cc}
    \text{Find }u\in\scU \text{ such that:}      \\
     b(u,v)=\ell(v) \qquad \text{for every } v\in\scV,
    \end{array}
    \right.
    \label{eq:variational_abstract}
\end{equation}
for any continuous linear functional $\ell\in\scV'$ that vanishes on the subspace $\scV_0$, given by
\begin{equation}
    \scV_0=\{v\in\scV:\ b(u,v)=0\ \forall u\in\scU\}.
    \label{eq:transpose_nullspace}
\end{equation}
\begin{remark}
In the problem at hand, as it will be explained below, this subspace is trivial, thereby $\ell$ needn't be restricted, delivering a broader scope to the notion of well-posedness. Particularly, the existence, uniqueness and continuous data dependence may extend beyond the specifications in assumptions 1 and 2, because the dual space $\scV'$ might include continuous functionals outside those limitations, and even then we are covered by Babu\v{s}ka-Ne\v{c}as' result.    
\end{remark}

Let us define the functional spaces and bilinear forms for the problems we have introduced
\paragraph{Classical weak formulation for the concentration problem:}
\begin{itemize}
    \item \emph{Trial space:} $\scU^{\scC}_{c}:=H^1(\Omega)\ni c$.
    \item \emph{Test space:}  $\scV^{\scC}_{c}:=H^1(\Omega)\ni r$.
    \item \emph{Bilinear form:} $b^{\scC}_{c}:\scU^{\scC}_{c} \times \scV^{\scC}_{c}$, with
    $$
    b^{\scC}_{c}(c,r):=\left(c,r\right)_\Omega + \dt\left(D\grad c,\grad r\right)_\Omega\,.
    $$
\end{itemize}
\paragraph{Classical weak formulation for the potential problem:}
\begin{itemize}
    \item \emph{Trial space:} $\scU^{\scC}_{\phi}:=H^1_{\Gamma_D}(\Omega)\ni \phi$.
    \item \emph{Test space:}  $\scV^{\scC}_{\phi}:=H^1_{\Gamma_D}(\Omega)\ni \zeta$.
    \item \emph{Bilinear form:} $b^{\scC}_{\phi}:\scU^{\scC}_{\phi} \times \scV^{\scC}_{\phi}$, with
    $$
    \left(\kappa\grad \phi,\grad \zeta\right)_\Omega + \langle \beta\phi, \zeta \rangle_{\Gamma_R}\,.
    $$
\end{itemize}
\paragraph{Mixed variational formulation for the concentration problem:}
\begin{itemize}
    \item \emph{Trial space:} $\scU^{\scM}_{c}:=H^1(\Omega)\times\bsL^2(\Omega)\ni (c,\bsj)$.
    \item \emph{Test space:}  $\scV^{\scM}_{c}:=H^1(\Omega)\times\bsL^2(\Omega)\ni (r,\bsw)$.
    \item \emph{Bilinear form:} $b^{\scM}_{c}:\scU^{\scM}_{c} \times \scV^{\scM}_{c}$, with
    $$
    b^{\scM}_{c}((c,\bsj),(r,\bsw)):=\left(c,r\right)_\Omega + \dt\left(D\grad c,\grad r\right)_\Omega
    \left(\grad c,\bsw\right)_\Omega + \left(D^{-1}\bsj, \bsw\right)_\Omega \,.
    $$
\end{itemize}
\paragraph{Mixed variational formulation for the potential problem:}
\begin{itemize}
    \item \emph{Trial space:} $\scU^{\scM}_{\phi}:=H^1_{\Gamma_D}(\Omega)\times\bsL^2(\Omega)\ni (\phi,\bsi)$.
    \item \emph{Test space:}  $\scV^{\scM}_{\phi}:=H^1_{\Gamma_D}(\Omega)\times\bsL^2(\Omega)\ni (\zeta,\bsv)$.
    \item \emph{Bilinear form:} $b^{\scM}_{\phi}:\scU^{\scM}_{\phi} \times \scV^{\scM}_{\phi}$, with
    $$
    b^{\scM}_{\phi}((\phi,\bsi),(\zeta,\bsv)):=\left(\kappa\grad \phi,\grad \zeta\right)_\Omega + \langle \beta\phi, \zeta \rangle_{\Gamma_R}
    \left(\grad \phi,\bsv\right)_\Omega + \left(\kappa^{-1}\bsi, \bsv\right)_\Omega \,.
    $$
\end{itemize}

We now show that both the classical formulation and the mixed one must be simultaneously well-posed or ill-posed. 
\begin{lemma}
\label{thm:mutually_well-posed}
    The following statements hold true:
    \begin{enumerate}[label=(\roman*)]
        \item problem \eqref{eq:conce_classical_weak} is well-posed if and only if problem \eqref{eq:conce_mixed_weak} is well-posed;
        \item problem \eqref{eq:poten_classical_weak} is well-posed if and only if problem \eqref{eq:poten_mixed_weak} is well-posed.
    \end{enumerate}
\end{lemma}
\begin{proof}
    First, consider the ``only if'' implication of (i). Take the unique solution $c\in H^1(\Omega)$, and recall that $\grad c \in \bsL^2(\Omega)$. Define $\bsj=-D\grad c$, that becomes uniquely determined too, and solves the second equation of \eqref{eq:conce_mixed_weak}. Since $D^{-1}\|\bsj\|=\|\grad c\|\leq\|c\|_1$, then the $\bsL^2$ norm of $\bsj$ is continuously controlled by the same data-dependent expression that bounds $\|c\|_1$. The same reasoning is applicable to (ii). Next, we prove the ``if'' direction in (ii). By hypothesis, $\phi\in H^1_{\Gamma_D}(\Omega)$ and $\bsi\in\bsL^2(\Omega)$ are the unique solutions for \eqref{eq:poten_mixed_weak}, which verify
    \begin{equation*}
    \|\phi\|_1 + \| \bsi\| \leq C \left(\|\bsS\|+\| I \|_{-1/2,\Gamma_N} + \| R \|_{-1/2,\Gamma_R}\right). 
    \end{equation*}
    Since the constitutive equation $\bsi=-\grad\phi-\bsS$ is strongly satisfied (i.e., in $\bsL^2(\Omega)$), then we can substitute for the flux in the balance equation, such that for each $\zeta$ there holds
    $$
    \langle \beta \phi,\zeta \rangle_{\Gamma_R} - 
    (\underbrace{-\grad\phi-\bsS}_{\bsi},\grad \zeta)_\Omega \;=\; 
        - \langle I, \zeta \rangle_{\Gamma_N} - \langle R, \zeta \rangle_{\Gamma_R}\;.
    $$
    After rearranging terms, it is clear that we retrieve the classical variational formulation \eqref{eq:poten_classical_weak}, so that $\phi$ is a solution for it as well. 
    The continuous dependence is automatically given by the hypothesis, just by dropping the $\|\bsi\|$ term on the right-hand side of the last bound. 
\end{proof}

Combining the well-posedness of the classical formulations, i.e., lemmas \ref{thm:conce_classical_well-posed} and \ref{thm:poten_classical_well-posed}, with the mutual well-posedness of those and the mixed formulations, namely, lemma \ref{thm:mutually_well-posed}, and the triviality of $\scV^0$ for each case, the immediate consequence is stated in the following corollary.
\begin{corollary}
\label{thm:mixed_well-posedness}
    Let assumptions 1 and 2 hold. Then, problems \eqref{eq:conce_mixed_weak} and \eqref{eq:poten_mixed_weak} are well-posed.
\end{corollary}
\begin{proof}
    The only fact still unproven is the triviality of $\scV^0$. In the classical formulations, the triviality of such a nullspace is a direct consequence of the bilinear form's coercivity. Let $v\in\scV_0$, so that $b(u,v)$ vanishes for every trial $u$, including $c=v$ thanks to symmetry. As a result, $0=b(v,v)\geq \alpha \| v \|^2_{1}$, which is only possible if $v\equiv 0$. Since coercivity was proven for both problems from assumptions 1 and 2, we now have full well-posedness of these classical weak problems in the sense of Babu\v{s}ka-Ne\v{c}as. By lemma \ref{thm:mutually_well-posed}, the mixed formulations are well-posed too.
\end{proof}
\section{Broken variational formulations and discretization}
\subsection{Abstract setting}
\label{sec:abstract_setting_broken}
We introduce another Hilbert space into the framework, the \emph{trace} trial space $\hat{\scU}$. Let $\scW$ be a superspace of $\scV$, with a possibly weaker norm $\|\cdot\|_{\scW}$ that will be transferred to $\scV$ as well. Let us extend bilinear form $b$ to $\scU\times\scW$ and assume that there is a new bilinear functional $\hat{b}:\hat{\scU}\times\scW\to\R$. We can so formulate a new abstract variational formulation
\begin{equation}
    \left\{ \begin{array}{cc}
    \text{Find }(u,\hat{u})\in\scU\times\hat{\scU} \text{ such that:}      \\
     b(u,v)+\hat{b}(\hat{u},v)=\ell(v) \qquad \text{for every } v\in\scW,
    \end{array}
    \right.
    \label{eq:variational_abstract_broken}
\end{equation}

The following assumptions and theorem are practically copied from \cite{BrokenForms15}, with a few minor adaptations.
\begin{assumption}
    The inf-sup condition \eqref{eq:infsup_abstract} holds, provided $\scV$ is equipped with norm $\|\cdot\|_{\scW}$
\end{assumption}
\begin{assumption}
    The spaces $\scV$, $\scW$ and $\hat{\scU}$ satisfy
    $$
    \scV = \{ v\in\scW:\ \hat{b}(\hat{u},v)=0\ \forall\hat{u}\in\hat{\scU})\},
    $$
    and there is a positive constant $\hat{\alpha}$ such that
    $$
    \hat{\alpha}\leq \sup\limits_{v\in\scW\setminus\{0\}} \frac{\hat{b}(\hat{u},v)}{\| v \|_{\scW}}\ 
    \forall\hat{u}\in\hat{\scU}).
    $$
\end{assumption}
The following is the key piece to obtain well-posedness of variational problems with broken test spaces. The proof can be seen in \cite[\S3]{BrokenForms15}.
\begin{theorem}
\label{thm:broken_well-posedness}
    Assumptions 3 and 4 imply
    \begin{equation*}
        \alpha_1 \|(u,\hat{u}) \|_{\scU\times\hat{\scU}} \leq 
        \sup\limits_{v\in\scW\setminus\{0\}}\frac{b(u,v)+\hat{b}(\hat{u},v)}{\|v\|_{\scW}},\ \forall (u,\hat{u})\in\scU\times\hat{\scU}.
    \end{equation*}
    Moreover, if $\scW_0=\{v\in\scW:\ b(u,v)+\hat{b}(\hat{u},v)=0\ \forall (u,\hat{u})\in\scU\times\hat{\scU} \}$,
    then $\scV_0=\scW_0$. Consequently, if $\scV_0$ is trivial, then \eqref{eq:variational_abstract_broken} is uniquely solvable, and the component $u$ of the solution coincides with the solution of \eqref{eq:variational_abstract}.
\end{theorem}

\subsection{Application to broken formulations of the concentration and potential equations}

Let $\mesh$ be a mesh or triangulation of $\Omega$, where each element is polygonal(or polyhedral) and convex, and the usual mesh regularity criteria are fulfilled. Let us now define the broken or product $H^1$ space:
$$
H^1(\mesh):=\{u\in L^2(\Omega):\ \forall K\in\mesh\ u|K\in H^1(\Omega) \}
\cong \prod\limits_{K\in\mesh} H^1(K),
$$
with norm $\| u \|^2_{1,\mesh} := \sum\limits_{K\in\mesh} \| u|_K\; \|^2_{1,K}$. 
Now, if $\tr_{\nml}^K$ is the normal trace operator of $H(\div,K)$ functions onto the element boundary $\bdry K$, let $\tr^{\mesh}_{\nml}:H(\div,\Omega)\to \prod\limits_{K\in\mesh} H^{-1/2}(\bdry K)$ be the skeleton normal trace operator. The range of this operator defines the space
$$
H^{-1/2}(\bdry\mesh):=\tr^{\mesh}_{\nml} H(\div,\Omega).
$$
We require a final bilinear operation on $H^{-1/2}(\bdry\mesh) \times H^1(\mesh)$:
$$
\langle \hat{q}_n, u \rangle_{\bdry\mesh}:= \sum\limits_{K\in\mesh} \langle \tr_{\nml}^K(\bsq|_K), (u)_K \rangle_{\bdry K} ,
$$
where $\bsq$ is any $H(\div,\Omega)$ extension of $\hat{q}_n$.
Notice that, elementwise, the pairing $\langle \tr_{\nml}^K(\bsq|_K), u_K \rangle_{\bdry K}$ makes sense as a duality pairing, because $\tr_{\nml}^K(\bsq|_K) \in H^{-1/2}(\bdry K)$ and $u_K \in H^1(K)$, so that its trace lies in $H^{1/2}(\bdry K)$.

At last, we need an interesting result developed in \cite[Thm. 2.3]{BrokenForms15}, where it is shown that for any $\hat{\sigma}_n\in H^{-1/2}(\bdry\mesh)$,
\begin{equation}
    \label{eq:skeleton_dual_norm}
    \| \hat{\sigma}_n \|_{-1/2,\bdry\mesh} = 
    \sup\limits_{u\in H^1(\mesh)\setminus \{0\}} 
    \frac{\langle \hat{\sigma}_n , u \rangle_{\bdry\mesh}}{\| u \|_{1,\bdry\mesh}}.
\end{equation}

With the new tools at hand, we postulate the \emph{broken mixed variational formulation} for the concentration problem.
\begin{equation}
\left\{
\begin{array}{ll}
\text{Find }c\in H^1(\Omega),\;\bsj\in \bsL^2(\Omega),\;\hat{j}_{n}\in H^{-1/2}(\bdry\mesh)
\text{ such that:}      \\
        \left(c,r\right)_\Omega - \dt\left(\bsj,\grad r\right)_\Omega 
        +\dt\langle \hat{j}_{n} , r \rangle_{\bdry\mesh}  \;=\; 
        \left(c_\prev,r\right)_\Omega - \dt \langle J, r \rangle_{\bdry\Omega}
        \qquad \text{for every } r\in H^1(\mesh)    \\
        \left(\grad c,\bsw\right)_\Omega + \left(D^{-1}\bsj, \bsw\right)_\Omega \;=\; 
        0  \qquad \text{for every } \bsw\in \bsL^2(\Omega)
\end{array}
\right.
\label{eq:conce_mixed_weak_broken}
\end{equation}
Similarly, the potential counterpart reads as follows:
\begin{equation}
\left\{
\begin{array}{ll}
\text{Find }\phi\in H^1_{\Gamma_D}(\Omega),\;\bsj\in \bsL^2(\Omega),\;\hat{i}_{n}\in H^{-1/2}(\bdry\mesh)
\text{ such that:}      \\
        \langle \beta \phi,\zeta \rangle_{\Gamma_R} - \left(\bsi,\grad \zeta\right)_\Omega 
        +\langle \hat{i}_{n} , \zeta \rangle_{\bdry\mesh} \;=\; 
        - \langle I, \zeta \rangle_{\Gamma_N} - \langle R, \zeta \rangle_{\Gamma_R}
        \qquad \text{for every } \zeta\in H^1(\mesh)    \\
        \left(\grad \phi,\bsv\right)_\Omega + \left(\kappa^{-1}\bsi, \bsv\right)_\Omega \;=\; 
        -\left(\kappa^{-1}\bsS,\bsv\right)_\Omega  \qquad \text{for every } \bsv\in \bsL^2(\Omega)
\end{array}
\right.
\label{eq:poten_mixed_weak_broken}
\end{equation}

The new ingredients present in \eqref{eq:conce_mixed_weak_broken} and \eqref{eq:poten_mixed_weak_broken} are once again organized and named consistently with the notation of subsection \ref{sec:abstract_setting_broken}.
\paragraph{Broken mixed variational formulation for the concentration problem:}
\begin{itemize}
    \item \emph{Field trial space:} $\scU^{\scB}_{c}:=H^1(\Omega)\times\bsL^2(\Omega)\ni (c,\bsj)$.
    \item \emph{Trace trial space:} $\hat{\scU}^{\scB}_{c}:=H^{-1/2}(\bdry\mesh)\ni \hat{j}_n$.
    \item \emph{Broken test space:} $\scW^{\scB}_{c}:=H^1(\Omega)\times\bsL^2(\Omega)\ni (r,\bsw)$.
    \item \emph{Bilinear forms:} $b^{\scB}_{c}:\scU^{\scB}_{c} \times \scW^{\scB}_{c}$, and 
    $\hat{b}^{\scB}_{c}:\hat{\scU}^{\scB}_{c} \times \scW^{\scB}_{c}$, with
    $$
    b^{\scB}_{c}((c,\bsj),(r,\bsw)):=\left(c,r\right)_\Omega + \dt\left(D\grad c,\grad r\right)_\Omega
    \left(\grad c,\bsw\right)_\Omega + \left(D^{-1}\bsj, \bsw\right)_\Omega
    \,,\quad
    \hat{b}^{\scB}_{c}(\hat{j}_n,(r,\bsw)):=\dt\langle \hat{j}_{n} , r \rangle_{\bdry\mesh}\,.
    $$
\end{itemize}
\paragraph{Broken mixed variational formulation for the potential problem:}
\begin{itemize}
    \item \emph{Field trial space:} $\scU^{\scB}_{\phi}:=H^1_{\Gamma_D}(\Omega)\times\bsL^2(\Omega)\ni (\phi,\bsi)$.
    \item \emph{Trace trial space:} $\hat{\scU}^{\scB}_{c}:=H^{-1/2}(\bdry\mesh)\ni \hat{i}_n$.
    \item \emph{Broken test space:} $\scW^{\scB}_{\phi}:=H^1_{\Gamma_D}(\Omega)\times\bsL^2(\Omega)\ni (\zeta,\bsv)$.
    \item \emph{Bilinear forms:} $b^{\scB}_{\phi}:\scU^{\scB}_{\phi} \times \scW^{\scB}_{\phi}$, and 
    $\hat{b}^{\scB}_{\phi}:\hat{\scU}^{\scB}_{\phi} \times \scW^{\scB}_{\phi}$, with
    $$
    b^{\scB}_{\phi}((\phi,\bsi),(\zeta,\bsv)):=\left(\kappa\grad \phi,\grad \zeta\right)_\Omega + \langle \beta\phi, \zeta \rangle_{\Gamma_R}
    \left(\grad \phi,\bsv\right)_\Omega + \left(\kappa^{-1}\bsi, \bsv\right)_\Omega
    \,,\quad
    \hat{b}^{\scB}_{\phi}(\hat{i}_n,(\zeta,\bsv)):=\langle \hat{i}_{n} , \zeta \rangle_{\bdry\mesh}\,.
    $$
\end{itemize}

We postulate as a theorem the well-posedness of the broken formulations just assembled.
\begin{theorem}
    Let assumptions 1 and 2 hold. Then, the broken mixed formulations \eqref{eq:conce_mixed_weak_broken} and \eqref{eq:poten_mixed_weak_broken} are well-posed.
\end{theorem}
\begin{proof}
    We must show that the hypotheses of theorem \ref{thm:broken_well-posedness} are true. Assumption 3 is the subject of corollary \ref{thm:mixed_well-posedness}, which is equivalent to having the corresponding inf-sup conditions for the mixed formulations. An important step requires making clear that the broken test norm $\| \cdot\|_{\scW}=\| \cdot \|_{1,\mesh}$ coincides with $\| \cdot\|_{\scV}=\| \cdot \|_{1}$ when restricted to ``unbroken'' $H^1(\Omega)$ functions:
    $$
    \alpha\| u\|_{\scU} \leq 
    \sup\limits_{v\in\scV\setminus\{0\}}\frac{b(u,v)}{\|v\|_{\scV}} =
    \sup\limits_{v\in\scV\setminus\{0\}}\frac{b(u,v)}{\|v\|_{\scW}}.
    $$
    Satisfaction of assumption 4 is achieved by definition of the bilinear form on $\hat{\scU}\times\scW$ given by the duality pairing $\langle \cdot, \cdot\rangle_{\bdry\mesh}$. Consider the potential problem, so that $\langle \hat{i}_n, \phi \rangle_{\bdry\mesh}$ is decomposed elementwise into $\langle \bsq, (\phi)_K \rangle_{\bdry\mesh}$. But, when using and ``unbroken'' $\phi\in H^1_{\Gamma_D}(\Omega)$, $(\phi)_K$ can be identified as the restriction $\phi|_K$. The zero trace on boundary part $\Gamma_D$ removes the contribution of the mesh facets that lie on that boundary. Moreover, it is well known that $H^1$ functions do not possess jumps across interfaces of adjacent elements, and since the normal trace of $\bsq$ changes sign between adjacent elements, the total sum vanishes. The inf-sup condition on the traces is given directly by \eqref{eq:skeleton_dual_norm}, with equality rather than inequality. Finally, the triviality of $\scV_0$ was proven above. Consequently, all of the requisites are ready and theorem \ref{thm:broken_well-posedness} guarantees the well-posedness of our broken mixed variational formulations.
\end{proof}

\section{Discretization with the Discontinuous Petrov-Galerkin method}

The Discontinuous Petrov-Galerkin method (DPG) is a finite element framework that, given a conforming finite-dimensional trial subspace $\scU^h\subset \scU$, guarantees discrete stability by choosing an appropriate conforming discrete test space $\scW^h\subset\scW$ \cite{demkowicz2015encyclopedia}. By Babu\v{s}ka's theorem \cite{babuvska1971error}, a Petrov-Galerkin method that verifies a discrete inf-sup condition is stable, that is:
$$
\| u - u^h\| \leq C \inf\limits_{w^h\in \scU^h} \| u - w^h\|_{\scU},
$$
where $C>0$ is a constant independent of the mesh size, and $u$ is the solution to the original problem and $u^h$ is the solution to the discrete problem
\begin{equation}
    \left\{ \begin{array}{cc}
    \text{Find }u^h\in\scU \text{ such that:}      \\
     b(u^h,v^h)=\ell(v^h) \qquad \text{for every } v^h\in\scW^h,
    \end{array}
    \right.
    \label{eq:variational_abstract_discrete}
\end{equation}
The issue is that the discrete inf-sup condition does not follow from the continuous one. For details on the DPG strategy for attaining this, we refer the reader to \cite{demkowicz2010class,demkowicz2011class,demkowicz2011analysis,DPGOverview}, among others. We skip directly to the formulations.

\paragraph{Discrete version of the broken mixed variational formulation for the concentration problem.}
\begin{equation}
\left\{
\begin{array}{ll}
   \text{Find }(c^h,\bsj^h,\hat{j}_n^h)\in \scU^{\scB,h}_{c}\times \hat{\scU}^{\scB,h}_{c})
\text{ such that:}      \\
        b^{\scB}_{c}((c^h,\bsj^h),(r^h,\bsw^h)) + \hat{b}^{\scB}_{c}(\hat{j}_n^h,(r^h,\bsw^h))
         \;=\; 
        \ell((r^h,\bsw^h))  \qquad \text{for every } (r^h,\bsw^h)\in \scW^{\scB,h}_{c}
\end{array}
\right. 
\label{eq:conce_mixed_broken_discrete}
\end{equation}

\paragraph{Discrete version of the broken mixed variational formulation for the potential problem.}
\begin{equation}
\left\{
\begin{array}{ll}
   \text{Find }(\phi^h,\bsi^h,\hat{i}_n^h)\in \scU^{\scB,h}_{\phi}\times \hat{\scU}^{\scB,h}_{\phi})
\text{ such that:}      \\
        b^{\scB}_{\phi}((\phi^h,\bsi^h),(\zeta^h,\bsv^h)) + \hat{b}^{\scB}_{\phi}(\hat{i}_n^h,(\zeta^h,\bsv^h))
         \;=\; 
        \ell((\zeta^h,\bsv^h))  \qquad \text{for every } (\zeta^h,\bsv^h)\in \scW^{\scB,h}_{\phi}
\end{array}
\right. 
\label{eq:poten_mixed_broken_discrete}
\end{equation}

Notice that in the previous two problems, we are taking test functions living in the broken space. That is, variational problems \eqref{eq:conce_mixed_broken_discrete} and \eqref{eq:poten_mixed_broken_discrete} are conforming discretizations of \eqref{eq:conce_mixed_weak_broken} and \eqref{eq:poten_mixed_weak_broken}, respectively.

Consider an integer $p\geq 1$, which is the nominal order of the finite element spaces. We use globally continuous, piecewise polynomials of degree $p$ to discretize the $H^1$ trial variables. The $\bsL^2$ trial variables are discontinous vector fields, piecewise polynomials of degree $p-1$. The discrete space for the normal traces, $\hat{\scU}^{h}\subset H^{-1/2}(\bdry\mesh)$, is obtained as the range of trace operator $\tr^{\mesh}_{\nml}$ applied to the Raviart-Thomas space of degree $p-1$ (i.e., it consists of discontinuous functions defined on the facets of the elements, piecewise polynomials of degree $p-1$). The construction of the shape functions may follow the details in \cite{Fuentes2015,demkowicz2023mathematical}. For sufficiently regular solutions, the choice of these spaces guarantee
$$
\inf\limits_{w^h\in \scU^h} \| u - w^h\|_{\scU} \leq C(u) h^p,
$$
where the constant $C(u)$ is independent of the mesh size $h$. This approximability estimate, along with the stability provided by the DPG methodology, ensures convergence.

The construction of the discrete space requires a special procedure. We denote by $B:\scU\to\scW'$ the action of bilinear functional $b(\cdot,\cdot)$ on any $u\in\scU$, turning it into a continuous linear functional over $\scW$. Let $\mcR:\scW\to\scW'$ be the Riesz map of Hilbert space $\scW$, then we define the \emph{trial-to-test} mapping, $\bfT:=\mcR^{-1}\circ B$. With this mapping, we set $\scW^h=\bfT(\scU^h)$, and it can be shown that the discrete inf-sup follows directly from this space choice \cite{QiuFortin}. Our test space, has an $\bsL^2$ component, whose Riesz map can be inverted, but the other component is a broken $H^1$ space. Inverting the Riesz map is unfeasible for that part of our test space. For that reason, for spaces other than $L^2$ we approximate the Riesz map with that of an \emph{enriched} finite dimensional subspace $\scW^{\enr}\subset\scW$.In practice, we choose the enriched space to be a space of polynomials of degree $p+\Delta p$. Below, we assume that $\Delta p$ is chosen so that it fulfills the final assumption.
\begin{assumption}
    The enriched test space $\scW^{\enr}$ is such that there exists a bounded linear operator $\Pi_F:\scW\to\scW^{\enr}$, known as \emph{Fortin} operator, with the property
    $$
    b(u_h,\Pi_F v_h - v_h) + \hat{b}(\hat{u_h},\Pi_F v_h - v_h)=0\ \forall (u,\hat{u})\in\scU\times\hat{\scU}.
    $$
    The value of the operator norm for $\Pi_F$ is given by $C_F>0$, that is, $C_F$ is the smallest number such that
    $$
    \| \Pi_F v \|_{\scW} \leq C_F \| v \|_{\scW}\quad\forall v\in\scW.
    $$
\end{assumption}
\begin{remark}
    An extensive work on studying Fortin operators applicable to the DPG method has been published to date. We refer, for instance, to \cite{QiuFortin,Nagaraj2015,BrokenForms15,demkowicz2020construction}. A rule of thumb would say that for quadrilaterals or hexahedra, $\Delta p=1$ may suffice, while for simplicial elements $\Delta p \geq d$ is frequently required.
\end{remark}

The optimal discrete test functions $\bsw^h$ and $\bsv^h$ in \eqref{eq:conce_mixed_broken_discrete} and \eqref{eq:poten_mixed_broken_discrete}, can be built thanks to the FOSLS (First-Order-System Least-Squares) method \cite{bochev2004least,bochev2009least}. Following that idea, we set
\begin{equation}
    \bsw^h=\grad \delta c^h    + D^{-1} \delta\bsj^h,\qquad\qquad 
    \bsv^h=\grad \delta \phi^h + \kappa^{-1} \delta\bsi^h,
    \label{eq:optimal_L2_test}
\end{equation}
where $(\delta c^h , \delta \bsj)\in \scU^{\scB,h}_{c}$ and $(\delta \phi^h , \delta \bsi)\in \scU^{\scB,h}_{\phi}$.
On the other hand, the remaining components of the discrete test spaces must be found through the solution of another variational problem.
\begin{equation}
\begin{aligned}
    ( r^{h} , r^{\enr})_{H^1(\mesh)} &=
    b^{\scB}_{c}((c^h,\bsj^h),(r^{\enr},0)) + \hat{b}^{\scB}_{c}(\hat{j}_n^h,(r^{\enr},0)), 
    \ \forall (r^{\enr},0)\in \scW^{\scB,\enr}_{c}\;, \\
    ( \zeta^{h},\zeta^{\enr})_{H^1(\mesh)} &=
    b^{\scB}_{\phi}((\phi^h,\bsi^h),(\zeta^h,\bsv^h)) + \hat{b}^{\scB}_{\phi}(\hat{i}_n^h,(\zeta^h,\bsv^h))
    \ \forall (\zeta^{\enr},0)\in \scW^{\scB,\enr}_{\phi}\;.
\end{aligned}
    \label{eq:optimal_H1_test}
\end{equation}
Fortunately, these variational problems are possible to implement locally (at each element $K\in\mesh$). We close this derivations by stating the theorem that encompasses the ultimate result.
\begin{theorem}
    \label{thm:convergence}
    Let assumptions 1, 2 and 5 hold. Suppose we are given a mesh $\mesh$ of $\Omega$ with maximum element size $h$, and that we choose conforming finite-dimensional trial spaces 
    $\scU^{\scB,h}_{c} \times\hat{\scU}^{\scB,h}_{c}$ and $\scU^{\scB,h}_{\phi} \times \hat{\scU}^{\scB,h}_{\phi}$, consisting in polynomial spaces of nominal order $p\geq 1$. Then, if the test spaces $\scW^{\scB,h}_{c}$ and $\scW^{\scB,h}_{c}$ are the span of the functions computed by procedures \eqref{eq:optimal_L2_test} and \eqref{eq:optimal_H1_test}, then the problems \eqref{eq:conce_mixed_broken_discrete} and \eqref{eq:poten_mixed_broken_discrete} are well-posed, and convergent to the exact solutions of \eqref{eq:conce_mixed_weak_broken} and \eqref{eq:poten_mixed_weak_broken}, respectively, with order of convergence $\mcO(h^p)$.
\end{theorem}

\section{Conclusions and future directions}
A pair of BVPs posed on the same domain are extensively studied, which are surrogates of the more complicated equations that govern the electrochemical phenomena within a battery. A sequence of implications has let us show well-posedness of advanced variational formulations, starting from simple assumptions on the materials and loads. This suggests that the models are general enough to perform appropriately when integrated in a dynamic, nonlinear solver. A similar result has already been registered, even with many more physical variables interacting with each other, but therein the discretization is carried out with a classical Bubnov-Galerkin method (see \cite{mora2024high}), which also uses a predictor-corrector time integration method, and simulates the three battery materials together. We plan on converging to a similar framework, but using the variational formulations and finite element method herein derived.

We emphasize on the fact that the Robin BC, a condition much less common than the Dirichlet and Neumann counterparts, plays a central role in the potential problem. In sections 2 through 4, analyzing the equation with that BC and obtaining all the related theoretical results, is a considerable contribution to the area of numerical treatment of PDEs.

The DPG method has been selected as the numerical method to approximately solve the broken mixed formulations. The fact that this method is compatible with higher order discretizations, and that the discrete stability issue is circumvented, make it an excellent candidate for becoming a computational and mathematical tool to advance in the science and engineering of batteries. Particularly, the possibility of driving automatic adaptive mesh refinements with its built-in error estimate (see \cite{demkowicz2012class,carstensen2014posteriori}) can be exploited to obtain better meshes for intricate and novel battery configurations. An exploration on its combination with optimization algorithms is in process \cite{mora2024optimization}.

\paragraph{Acknowledgments.}
The author thanks the institutional research funding by \emph{Fundación Universitaria Konrad Lorenz}, in its 2023 and 2024 calls.

\bibliographystyle{apalike}
\bibliography{main}
\end{document}